\documentclass{amsart}

\usepackage{amsfonts}
\usepackage{amssymb}
\usepackage{xy}
\input xy
\xyoption{all}
\usepackage{amsmath}
\usepackage{amsthm}
\usepackage{url}
\usepackage{verbatim}
\usepackage{hyperref}
\usepackage{cleveref}

\title{$THH$ and base-change for Galois extensions of ring spectra}
\date{\today}
\author{Akhil Mathew}
\address{Harvard University, Cambridge, MA 02138}
\email{amathew@math.harvard.edu}

\newcommand{\clg}{\mathrm{CAlg}}
\newcommand{\perf}{\mathrm{Perf}}
\newcommand{\fun}{\mathrm{Fun}}
\newcommand{\md}{\mathrm{Mod}}

\renewcommand{\sp}{\mathrm{Sp}}
\renewcommand{\hom}{\mathrm{Hom}}

\renewcommand{\ell}{\mathrm{Ell}}
\newcommand{\e}[1]{\mathbb{E}_{#1}}

\newtheorem{lemma}{Lemma}[section]

\newtheorem{corollary}[lemma]{Corollary}
\newtheorem{theorem}[lemma]{Theorem}
\newtheorem*{thm}{Theorem}
\newtheorem{proposition}[lemma]{Proposition}

\theoremstyle{definition}
\newtheorem{definition}[lemma]{Definition}
 \theoremstyle{definition}

\newtheorem{example}[lemma]{Example}

\newtheorem*{question}{Question}

\begin{document}

\begin{abstract}
We treat the question of base-change in $THH$ for faithful Galois extensions
of ring spectra in the sense of Rognes.
Given 
a faithful Galois extension $A \to B$ of ring spectra, we consider whether the
map $THH(A) \otimes_A B \to THH(B)$ is an equivalence. 
We reprove and extend positive results of Weibel-Geller and McCarthy-Minasian
and offer new examples of Galois extensions for which base-change holds. 
We also provide a counterexample where base-change fails. 
\end{abstract}

\maketitle
\section{Introduction}

Let $R$ be an $\mathbb{E}_1$-ring spectrum. The \emph{topological Hochschild
homology} $THH(R)$ of $R$ is a spectrum constructed as  the geometric
realization of a certain cyclic object built from $R$, a homotopy-theoretic
version of the Hochschild 
complex of an associative ring. 
Topological Hochschild homology has been studied in particular because of its connections with
algebraic $K$-theory via the theory of trace maps. 
More generally, if $R$ is an $\mathbb{E}_1$-algebra in $A$-modules for an
$\mathbb{E}_\infty$-ring $A$,
then one can define a relative version $THH^A(R)$.

In \cite{WG}, it is shown that Hochschild homology for
commutative rings satisfies an \'etale base-change result. Equivalently, if $k$ is
a commutative ring and  if $A \to
B$ is an \'etale morphism of commutative $k$-algebras with $A$ flat over $k$, then there is a natural
equivalence $$B
\otimes_A THH^{k}(A) \simeq
THH^{k}(B).$$ Weibel-Geller's result also applies in the non-flat case,
although it cannot be stated in this manner.  

One can hope to generalize the Weibel-Geller result to the setting of ring
spectra. 
This leads to the following general question.

\begin{question} 
Let $A \to B$ be a morphism of  $\mathbb{E}_\infty$-ring spectra.
When is the map \begin{equation} THH(A) \otimes_A B \to THH(B)
\label{THHcomparemap}\end{equation} an equivalence? 
\end{question}

Following Lurie, we will use the following definition 
of \'etaleness: 

\begin{definition} 
\label{lurieetale}
A morphism $ A\to B$ of $\mathbb{E}_\infty$-ring spectra 
is \emph{\'etale}
if
$\pi_0(A) \to \pi_0(B)$ is \'etale and the natural map $\pi_*(A)
\otimes_{\pi_0(A)} \pi_0(B) \to \pi_*(B)$ is an isomorphism. 
\end{definition}

In \cite{McM}, 
McCarthy-Minasian 
consider this question for an \'etale morphism\footnote{We note that
\cite{McM} use the word``\'etale'' differently in their paper.} of connective
$\mathbb{E}_\infty$-rings  and prove the analog of the
Weibel-Geller theorem, i.e., that \eqref{THHcomparemap} is an equivalence (cf. \cite[Lem. 5.7]{McM}). 
In fact, they prove the result more generally for any \emph{$THH$-\'etale
morphism} of connective 
$\mathbb{E}_\infty$-rings. 

In the setting of structured ring spectra, however, there 
are additional morphisms of nonconnective ring spectra that 
have formal properties similar to those of \'etale morphisms, though they are not \'etale on homotopy
groups. The \emph{faithful Galois extensions} of Rognes \cite{rognes} are key
examples here.

This note is primarily concerned with the following analog of the
Weibel-Geller and McCarthy-Minasian question. 

\begin{question} 
Let $A \to B$ be a $G$-Galois extension of $\mathbb{E}_\infty$-ring spectra,
with $G$ finite.
When is the comparison map 
\eqref{THHcomparemap} an equivalence? 
\end{question}

We make two main observations here. 
Our first observation uses the fact that $THH$, like algebraic $K$-theory, is an invariant not only of ring
spectra but of stable $\infty$-categories. 
We refer, for example, to \cite{BM1, BGT} for a treatment of $THH$ in this context. 
Using Galois descent, we observe that the map 
\eqref{THHcomparemap} is an equivalence if and only if the map 
\( THH(A) \to THH(B)^{hG}  \)
is an equivalence. 
These maps are the comparison maps for the \emph{Galois descent} problem in
$THH$. Consequently, the results of \cite{CMNN} provide numerous examples
in chromatic homotopy theory where
\eqref{THHcomparemap} is an equivalence. 

Our second observation is to reinterpret the base-change question 
for $THH$ in terms of the formulation $THH(R) \simeq S^1 \otimes R$ for
$\mathbb{E}_\infty$-rings, due to McClure-Schw{\"a}nzl-Vogt \cite{MSV}. 

As a result, we obtain an example where \eqref{THHcomparemap} is not an
equivalence. 

\begin{theorem} 
\label{mainexample}
There is a faithful $G$-Galois extension $A \to B$ of $\mathbb{E}_\infty$-ring spectra which is a
faithful $G$-Galois extension such that \eqref{THHcomparemap} is not an
equivalence.
\end{theorem} 

Our counterexample Galois extension is very simple; it is the map $C^*(S^1; \mathbb{F}_p) \to
C^*(S^1; \mathbb{F}_p)$ induced by the degree $p$ cover $S^1 \to S^1$. 

We in fact pinpoint exactly what goes wrong from a categorical perspective, and
why this phenomenon cannot happen in the \'etale setting, thus proving  a
variant 
of the Weibel-Geller-McCarthy-Minasian theorem in the non-connective setting: 

\begin{theorem} 
\label{ourversionWG}
Let $R$ be an $\mathbb{E}_\infty$-ring, and let $A \to B$ be an \'etale morphism of
$\mathbb{E}_\infty$-$R$-algebras (possibly nonconnective). Then the natural map $THH^R(A) \otimes_A B \to
THH^R(B)$ is an equivalence. 
\end{theorem} 
The use of categorical 
interpretation of $THH$ in proving such base-change theorems is not new; McCarthy-Minasian 
 use this interpretation in \cite{McM} in a different manner. 
\subsection*{Acknowledgments}
I would like to thank John Rognes and the referee for several helpful comments. 
The author is supported by the NSF Graduate Research Fellowship
under grant DGE-110640.
\section{Categorical generalities}
Let $\mathcal{C}$ be a cocomplete $\infty$-category, and let $x \in
\mathcal{C}$. 
Given $x \in \mathcal{C}$, we can \cite[\S 4.4.4]{HTT} construct an object $S^1 \otimes x$. 

Choose a basepoint $\ast \in S^1$. 
Then we have a diagram
\begin{equation} \label{s1square} \xymatrix{
x \ar[d] \ar[r] & y \ar[d] \\
S^1 \otimes x \ar[r] &  S^1 \otimes y
}.\end{equation}
As a result of this diagram,
we have a natural map
in $\mathcal{C}$,
\begin{equation} \label{naturalpushout} (S^1 \otimes x ) \sqcup_x y \to S^1 
\otimes y.
 \end{equation}
In order for \eqref{naturalpushout} to be an equivalence, for any object $z \in \mathcal{C}$, the square
of spaces
\begin{equation} \label{freeloopsquarehom} \xymatrix{
\hom(S^1, \hom_{\mathcal{C}}(y, z)) \ar[d] \ar[r] & \hom_{\mathcal{C}}(y, z) \ar[d] \\
\hom(S^1, \hom_{\mathcal{C}}(x, z)) \ar[r] &  \hom_{\mathcal{C}}(x,z)
}\end{equation}
must be homotopy cartesian. 
This happens only in very special situations.

\begin{proposition} 
\label{easycartlemma}
Let $f\colon X \to Y$ be a map of spaces. Then the diagram
\begin{equation} \label{generalcart} \xymatrix{
\hom(S^1, X) \ar[d] \ar[r] &  X \ar[d] \\
 \hom(S^1, Y) \ar[r] &  Y
}\end{equation}
is homotopy cartesian if and only if for every point $p \in X$, the map from
the connected component
of $X$ containing $p$
to that of $Y$ containing $f(p)$ is a homotopy equivalence. 
\end{proposition} 
\begin{proof} 
Without loss of generality, we may assume that $X, Y$ are connected spaces, 
In this case, choosing compatible basepoints in $X, Y$, we get equivalences
\[ \Omega X \simeq \mathrm{fib}\left( \hom(S^1, X) \to X\right), \quad 
\Omega Y \simeq \mathrm{fib} \left( \hom(S^1, Y) \to Y \right),
\]
and the fact that 
\eqref{generalcart}
is homotopy cartesian now implies that $\Omega X \to \Omega Y$ is a homotopy
equivalence. Since $X$ and $Y$ are connected, this implies that $X \to Y$ is a
homotopy equivalence. 
\end{proof} 

\begin{definition} 
We will say that a map of spaces $X \to Y$ is a \emph{split covering space} if the
equivalent conditions
of \Cref{easycartlemma} are met. In particular, $X \to Y$ is a covering space, which is
trivial on each connected component of $Y$.
\end{definition} 

Observe that the base-change of a split covering space is still a split
covering space. 

\begin{corollary} 
\label{whenbasechangeholds}
Suppose $x \to y$ is a morphism in $\mathcal{C}$ as above. Then the natural map
$(S^1 \otimes x) \sqcup_x y \to S^1 \otimes y$ is an equivalence if and only if,
for every object $z \in \mathcal{C}$, the induced map of spaces
\( \hom_{\mathcal{C}}(y, z) \to \hom_{\mathcal{C}}(x, z)  \)
is a split cover.
\end{corollary} 
\begin{proof} 
Our map is an equivalence if and only if 
\eqref{freeloopsquarehom} is homotopy cartesian for each $z \in \mathcal{C}$. 
By \Cref{easycartlemma}, we get the desired claim. 
\end{proof} 

We now give this class of morphisms a name.

\begin{definition} 
\label{strongly}
A morphism $x \to y$ in an $\infty$-category $\mathcal{C}$ is said to be
\emph{strongly 0-cotruncated} if, for every $z \in \mathcal{C}$, the map
$\hom_{\mathcal{C}}(y, z) \to \hom_{\mathcal{C}}(x, z)$ is a split covering
space.
\end{definition} 

\Cref{whenbasechangeholds} states that $x \to y$ has the property that $(S^1
\otimes  x) \sqcup_x y \to
S^1 \otimes y$ is an equivalence if and only if the map is strongly
0-cotruncated.

For passage 
to a relative setting, we will find the following useful.
\begin{proposition} 
\label{relativecotrunc}
Let $\mathcal{C}$ be a cocomplete $\infty$-category, let $a \in \mathcal{C}$,
and let $x \to y$ be a morphism in $\mathcal{C}_{a/}$. If $x \to y$ is strongly
0-cotruncated when regarded as a morphism in $\mathcal{C}$, then it is strongly
0-cotruncated 
when regarded as a morphism in $\mathcal{C}_{a/}$.
\end{proposition} 
\begin{proof} 
Suppose $a  \to z$ is an object of  $\mathcal{C}_{a/}$.
Then we have
\begin{gather*} \hom_{\mathcal{C}_{a/}}(y, z) 
= \mathrm{fib} \left( \hom_{\mathcal{C}}(y, z) \to \hom_{\mathcal{C}}(a, z)\right)
, \\
\hom_{\mathcal{C}_{a/}}(x, z) 
= \mathrm{fib} \left( \hom_{\mathcal{C}}(x, z) \to \hom_{\mathcal{C}}(a, z)\right)
.
\end{gather*}
Since $\hom_{\mathcal{C}}(y, z) \to \hom_{\mathcal{C}}(x,z)$ is 
a split cover, it follows easily that the same holds after taking homotopy fibers over
the basepoint in $\hom_{\mathcal{C}}(a, z)$.
In fact, we can assume without loss of generality that
$\hom_{\mathcal{C}}(x,z)$ is connected, in which case $\hom_{\mathcal{C}}(y,
z)$ is a disjoint union $\bigsqcup_S \hom_{\mathcal{C}}(x,y)$. 
Taking fibers over the map to $\hom_{\mathcal{C}}(a, z)$ preserves the disjoint union as desired, so the map on fibers
is a split cover. 
\end{proof}

\section{$\mathbb{E}_\infty$-ring spectra}

We let $\clg$ denote the $\infty$-category of $\mathbb{E}_\infty$-ring spectra. 
The construction $THH$ in this case can be interpreted (by \cite{MSV}) as
tensoring with $S^1$: that is, we have
\[ THH(A) \simeq S^1 \otimes A, \quad A \in \clg.  \]
If one works in a relative setting, under an $\mathbb{E}_\infty$-ring $R$, then 
one has $THH^{R}(A) \simeq S^1 \otimes A$, where the tensor product is
computed in $\clg_{R/}$.

Given a morphism in $\clg_{R/}$, $A \to B$, 
we can use the setup of the previous section and obtain a morphism 
\[ THH^R(A) \otimes_A B \to THH^R(B)  \]
which is a special case of \eqref{naturalpushout}. 
The base-change problem for $THH$ asks when this is an equivalence.

By \Cref{whenbasechangeholds}, this is equivalent to the condition that
the morphism $A \to B$ in
$\clg_{R/}$ should be strongly 0-cotruncated.
We can now prove \Cref{ourversionWG} from the introduction, which we restate
for convenience.

\begin{thm}[]
Let $R$ be an $\mathbb{E}_\infty$-ring and let $A \to B$ be  an \'etale
morphism (as in \Cref{lurieetale}) in $\clg_{R/}$. Then the natural morphism
\( THH^R(A) \otimes_A B \to THH^R(B)  \)
is an equivalence. 
\end{thm} 
This is closely related to \cite[Theorem 0.1]{WG}
and includes it in the case of a flat extension $R \to A$ of discrete
$\e{\infty}$-rings	. 
For connective $\mathbb{E}_\infty$-rings, this result is \cite[Lem. 5.7]{McM}
(who treat more generally the case of a $THH$-\'etale morphism).
\begin{proof} 
Given an \'etale morphism $A \to B$ in
$\clg_{R/}$, we need to argue that it is strongly 0-cotruncated. 
By \Cref{relativecotrunc}, we may reduce to the case where $R = S^0$. 
Given $C \in \clg$, we have a 
homotopy cartesian square
\[ \xymatrix{
\hom_{\clg}(B, C) \ar[d]  \ar[r] &  \hom_{\mathrm{Ring}} (\pi_0 B, \pi_0 C) \ar[d] \\
\hom_{\clg}(A, C) \ar[r] &  \hom_{\mathrm{Ring}} (\pi_0 A, \pi_0 C)
},\]
by, e.g.,  \cite[\S 7.5]{higheralg}. Here $\mathrm{Ring}$ is the category of
rings.
Since the right vertical map is a map of discrete spaces and therefore a
split covering, it follows that 
$\hom_{\clg}(B, C) \to \hom_{\clg}(A, C)$ is a split covering as desired.

\end{proof}

We also note in passing that the \'etale descent theorem has a partial
converse in the setting of \emph{connective} $\mathbb{E}_\infty$-rings. 
We note that this rules out non-algebraic Galois extensions. 
\begin{corollary} 
Let $A \to B$ be a morphism of connective $\mathbb{E}_\infty$-rings which is almost
of finite presentation \cite[\S 7.2.4]{higheralg}. Suppose the map
$THH(A) \otimes_A B \to THH(B)$ is an equivalence. Then $A \to B$ is \'etale.
\end{corollary}
\begin{proof} 
Indeed, $B$ defines a 0-cotruncated object (\Cref{0cotrunc}) of $\clg_{A/}$ 
and it is well-known that this, combined with the fact that $B$ is almost of
finite presentation, implies that $B$ is \'etale.  
We reproduce the argument for the convenience of the reader. 

In fact, 
since $B$ is 0-cotruncated, one finds that for any $B$-module $M$, the space of
maps\footnote{For an $\infty$-category $\mathcal{C}$ and a morphism $x \to y$,
we let $\mathcal{C}_{x//y}$ denote $(\mathcal{C}_{x/})_{/y}$ where $y \in
\mathcal{C}_{x/}$ via the given morphism.} 
$\hom_{\clg_{A// B}}(B, B \oplus M)$ is homotopy discrete, where the
$\mathbb{E}_\infty$-ring $B \oplus M$
is given the square-zero multiplication. 
Replacing $M$ by $\Sigma M$, it follows that 
\[  \hom_{\clg_{A// B}}(B, B \oplus M) \simeq 
\Omega \hom_{\clg_{A// B}}(B, B \oplus  \Sigma M)
\]
is actually contractible. 
Thus the cotangent complex $L_{B/A}$ vanishes, which implies that $B$ is
\'etale over $A$ by \cite[Lem. 8.9]{DAGVII}.
The connectivity is used in this last step. 
\end{proof} 

The above argument also appears in \cite[\S 9.4]{rognes}, where it is shown
that a map $A \to B$ which is 0-cotruncated as in \Cref{0cotrunc} below (which Rognes calls \emph{formally
symmetrically \'etale}, and which has been called \emph{$THH$-\'etale} in
\cite{McM}) has to have vanishing cotangent complex (which is called
\emph{$TAQ$-\'etale)}; see \cite[Lem.
9.4.4]{rognes}. The key point is that in the connective setting,
$TAQ$-\'etaleness plus a weak finiteness condition is enough to imply
\'etaleness. This entirely breaks down when one works with nonconnective
$\e{\infty}$-ring spectra.

\section{Connection with descent}

In this section, we will show that the question of base-change in $THH$ is
equivalent to a descent-theoretic question. We will then use some of the
descent results of \cite{CMNN} to obtain examples where base-change for $THH$
holds. 
Let $A \to B$ be a faithful $G$-Galois extension of $\mathbb{E}_\infty$-rings for $G$ a finite group.

To begin with, we will need to recall a fact about Galois descent.

\begin{proposition}[{Cf. for example \cite[Ch.~6]{meier} or
\cite[Th.~2.8]{Bannerjee} or \cite[Th.~9.4]{galoischromatic}}]
If $A \to B$ is a faithful $G$-Galois extension, then we have an equivalence
of symmetric monoidal $\infty$-categories
\[ \md
(A) \simeq \md(B)^{hG},  \]
where the left adjoint is extension of scalars along $A \to B$ and the right
adjoint is given by taking homotopy fixed points. 
\label{galdesc}
\end{proposition}

We can restate the above equivalence in the following manner.

\begin{corollary} 
Let $\fun(BG, \sp)$ be the symmetric monoidal $\infty$-category of $G$-spectra
equipped with a $G$-action. Then we have a natural equivalence
\[ \md_{\fun(BG, \sp)}(B) \simeq \md_{\sp}(A)  \]
given by taking homotopy fixed points. 
\end{corollary} 
\begin{proof} 
This follows from \Cref{galdesc} using the fact that the construction of
forming modules in a symmetric monoidal $\infty$-category is compatible with
homotopy limits of symmetric monoidal $\infty$-categories. 
\end{proof}

Let $\mathcal{C} = \mathrm{Fun}(BG, \clg)$ be the $\infty$-category of
$\mathbb{E}_\infty$-algebras equipped with a $G$-action, so that $B$ 
defines an object of $\mathcal{C}$. We have therefore have natural
equivalences of $\infty$-categories
\begin{equation} \label{htpyfixedpoints} \mathcal{C}_{B/} \simeq \clg(
\fun(BG, \sp))_{B/} 
\simeq \clg ( \md_{\fun(BG, \sp)}(B) )  \simeq \clg( \md(A)).
\end{equation} 
where the last equivalence is given by taking
homotopy fixed points. 
We now obtain: 

\begin{proposition} 
For a faithful $G$-Galois extension $A \to B$, the following two statements
are equivalent:
\begin{itemize}
\item  
$THH(A) \otimes_A B \to THH(B)$ is an
equivalence. 
\item $THH(A) \to THH(B)$ is a faithful $G$-Galois extension.
\item
The map $THH(A) \simeq \left( THH(A) \otimes_A B \right)^{hG} \to THH(B)^{hG}$
is an equivalence.
\end{itemize}
\label{descvsbasechange}
\end{proposition} 
\begin{proof} 
In this case, 
the maps $B \to THH(A) \otimes_A B \to THH(B)$ that we obtain are
$G$-equivariant, as they are natural in the $\mathbb{E}_\infty$-$A$-algebra $B$. 
Therefore, the map $THH(A) \otimes_A B \to THH(B)$ is naturally a morphism in
$\clg(\fun(BG, \sp))_{B/}$. 
By \eqref{htpyfixedpoints}, the map is an equivalence if and only if 
it induces an equivalence on homotopy fixed points. 

Finally, if $THH(A) \otimes_A B \to THH(B)$ is an equivalence, then the morphism 
$THH(A) \to THH(B)$ is a base-change of the faithful $G$-Galois extension $A
\to B$ and is thus a faithful $G$-Galois extension itself. 
Conversely, if $THH(A) \to THH(B)$ is a faithful $G$-Galois extension, then the
descent map $THH(A) \to THH(B)^{hG}$ is an equivalence. 
\end{proof}

In particular, the map $A \to B$ is strongly 0-cotruncated if and only if one
has \emph{Galois descent} for $THH$ along the map $A \to B$. 
In \cite{CMNN}, we give  a general criterion for 
proving descent in telescopically localized $THH$.

\begin{theorem}[\cite{CMNN}] 
\label{CMNNthm}
Suppose $A \to B$ is a $G$-Galois extension such that the map $K_0(B) \otimes
\mathbb{Q} \to K_0(A) \otimes \mathbb{Q}$ induced by restriction of scalars is
surjective. Fix an implicit prime $p$ and a height $n$. Fix a  
weakly additive (cf. \cite[Def.~3.11]{CMNN}) invariant $E$ of $\kappa$-compact small idempotent-complete $A$-linear
$\infty$-categories taking values in a presentable stable $\infty$-category.
Then the natural morphisms $$L_n^f E( \perf(A))
\to L_n^f E( \perf(B))^{hG} \to \left( L_n^f E(\perf(B))\right)^{hG}$$
are equivalences, where $L_n^f$ denotes finitary $L_n$-localization. 
In particular, one can take $E = K, THH, TC$. 
\end{theorem}

As a result, we can prove that the base-change map is an equivalence in a large
class of ``chromatic'' examples of Galois extensions.
\begin{theorem} 
Suppose $A \to B$ is a faithful $G$-Galois extension of $\mathbb{E}_\infty$-rings. 
Assume that for every prime $p$, the localization $A_{(p)}$ is $L_n^f$-local
for some $n = n(p)$.
Suppose the map $K_0(B)  \otimes
\mathbb{Q} \to K_0( A) \otimes \mathbb{Q}$ 
is surjective (or equivalently has image containing the unit). Then the
base-change map $THH(A) \otimes_A B \to THH(B)$ is an equivalence. 
\end{theorem} 
\begin{proof} 
To check that the map 
$THH(A) \otimes_A B \to THH(B)$ is an equivalence, it suffices to localize at
$p$, so we may assume $A$ and $B$ are $p$-local, and therefore $L_n^f$-local.
Since $L_n^f$ is a smashing localization, it follows that all $THH$ terms in
sight are automatically $L_n^f$-localized. 
In this case, the result follows by combining 
\Cref{descvsbasechange} and \Cref{CMNNthm}. 
\end{proof} 
\begin{example} 
Most classes of examples of faithful Galois extensions in chromatic homotopy
theory satisfy the conditions of \Cref{CMNNthm}. We refer to \cite[\S 5]{CMNN}
for a detailed treatment. 
For example:
\begin{enumerate}
\item  
The $C_2$-Galois extension $KO \to KU$ or the $C_{p-1}$-Galois
extension $L \to \widehat{KU}_p$.  
\item
The
$G$-Galois extension  $E_n^{hG} \to E_n$
if $G$ is a finite subgroup of the extended Morava stabilizer group (cf.
\cite[Appendix B]{CMNN} by Meier, Naumann, and Noel). 
\item 
Any Galois extension
of $TMF[1/n], Tmf_0(n)$ or related spectra.
\end{enumerate}
 It follows that the comparison map
in $THH$ is an equivalence for these Galois extensions. 
\end{example}

\section{A counterexample}

In this section, we will give an example over $\mathbb{F}_p$ where the
comparison (or  equivalently descent) map for $THH$ is not an equivalence. 
We begin with a useful weakening of \Cref{strongly}.
\begin{definition} 
\label{0cotrunc}
A morphism $x \to y$ in an $\infty$-category $\mathcal{C}$ is said to be
\emph{0-cotruncated} if, for every $z \in \mathcal{C}$, the map
$\hom_{\mathcal{C}}(y, z) \to \hom_{\mathcal{C}}(x, z)$ is a covering space
(i.e., has discrete homotopy fibers over any basepoint). 
An object $x \in \mathcal{C} $ is said to be \emph{0-cotruncated}
if $\hom_{\mathcal{C}}(x, z)$ is discrete for any $z \in \mathcal{C}$.
\end{definition} 

The condition that $x \to y$ should be cotruncated is equivalent to the
statement that $y \in \mathcal{C}_{x/}$ should define a 0-cotruncated object. 
Note that an object $x \in \mathcal{C}$ is 0-cotruncated if and only if the natural map
$x \to S^1 \otimes x$ is an equivalence.

In the setting of $\mathbb{E}_\infty$-ring spectra, \'etale morphisms are far from the
only examples of 0-cotruncated morphisms. 
For example, any faithful $G$-Galois extension in the sense of 
Rognes \cite{rognes} is 0-cotruncated. This is essentially \cite[Lemma
9.2.6]{rognes}. 
However, we show that faithful Galois extensions need not be \emph{strongly}
0-cotruncated. Equivalently, base-change for $THH$ can fail for them. 

\begin{proof}[Proof of \Cref{mainexample}]
Consider the degree $p$ map $S^1 \to S^1$, which is a $\mathbb{Z}/p$-torsor.
Let $k$ be a separably closed field of characteristic $p$.
For a space $X$, we let $C^*(X; k) = F(X_+; k)$ denote the
$\mathbb{E}_\infty$-rings of $k$-valued cochains on $X$.
The induced map of $\mathbb{E}_\infty$-rings 
$\phi\colon C^*(S^1; k) \to C^*(S^1; k)$ is a faithful
$\mathbb{Z}/p$-Galois extension of $\mathbb{E}_\infty$-ring spectra. 
This follows from \cite[Prop. 5.6.3(a)]{rognes} together with the criterion for
the faithfulness via vanishing of the Tate construction \cite[Prop.
6.3.3]{rognes}. See also
\cite[Th. 7.13]{galoischromatic}. 

We will show, nonetheless, that $\phi$ does not satisfy 
base-change for $THH$, or equivalently that it is not strongly 0-cotruncated.
It suffices to show this in $\clg_{k/}$ in view  of \Cref{relativecotrunc}.

By $p$-adic homotopy theory \cite{mandell} (see also \cite{DAG13}, which
does not assume $k = \overline{\mathbb{F}_p}$), the natural map
\[ S^1 \to \hom_{\clg_{k/}}( C^*(S^1; k), k)  \]
exhibits 
$\hom_{\clg_{k/}}( C^*(S^1; k), k) $ as the $p$-adic completion of $S^1$. 
In particular, 
$\hom_{\clg_{k/}}( C^*(S^1; k), k) \simeq K(\mathbb{Z}_p, 1) $ and the map
given by precomposition with $\phi$
$$\hom_{\clg_{k/}}( C^*(S^1; k), k) \stackrel{\phi^*}{\to} \hom_{\clg_{k/}}(
C^*(S^1; k), k),$$
is identified with multiplication by $p$, $K(\mathbb{Z}_p, 1) \to
K(\mathbb{Z}_p, 1)$. In particular, while this is a covering map, it is
\emph{not} a split covering map, so that $\phi$ is not strongly 0-cotruncated. 
\end{proof}

The use of cochain algebras in providing such counterexamples goes back to an
idea of Mandell \cite[Ex. 3.5]{McM}, who gives an example of a morphism of
$\e{\infty}$-ring spectra with trivial cotangent complex (i.e., is
$TAQ$-\'etale) which is not
$THH$-\'etale. Namely, Mandell shows that if $n  > 1$, then the map $C^*( K(
\mathbb{Z}/p, n); \mathbb{F}_p) \to 
\mathbb{F}_p$ has trivial cotangent complex. 

We close by observing that it is the fundamental group that it is at the root
of these problems. 
\begin{proposition} Let $X$ be a \emph{simply connected,} pointed space, and let $A \to B$ be a
faithful
$G$-Galois extension of $\mathbb{E}_\infty$-rings. 
In this case, the map
of $\mathbb{E}_\infty$-rings 
\[ (X \otimes A) \otimes_A B \to X \otimes B,  \]
is an equivalence. 
\end{proposition}

In particular, one does have base-change for higher
topological Hochschild homology (i.e., where $X = S^n, n > 1$). 

\begin{proof} 
Following the earlier reasoning, it suffices to show that whenever 
$C \in \clg$, the square
\[ \xymatrix{
\hom(X, \hom_{\clg}(B, C))
\ar[d] \ar[r] & \hom_{\clg}(B, C) \ar[d] \\
\hom(X, \hom_{\clg}(A ,C)) \ar[r] &  \hom_{\clg}(A, C)
}\]
is homotopy cartesian. 
However, this follows because $\hom_{\clg}(B, C) \to \hom_{\clg}(A, C)$ is a covering space,
and $X$ is simply connected. 
\end{proof}

\bibliographystyle{alpha}
\bibliography{THH}
\end{document}